\documentclass[10pt]{amsart}
\usepackage{SC}
\fullpage

\newtheorem*{lemLnstructure}{Lemma \ref{lem:Lnstructure}}

\begin{document}

\begin{abstract}
This paper is a continuation of \cite{zakharevich10}, in which we defined the notion of a polytope complex and its $K$-theory.  In this paper we produce formulas for the delooping of a simplicial polytope complex and the cofiber of a morphism of simplicial polytope complexes.  Along the way we also prove that the (classical and higher) scissors congruence groups of polytopes in a homogeneous $n$-manifold (with sufficient geometric data) are determined by its local properties.
\end{abstract}

\title{Simplicial Polytope Complexes and Deloopings of $K$-theory}
\author{Inna Zakharevich}
\date{}
\maketitle

\section{Introduction}

This paper is a continuation of the work started in \cite{zakharevich10}, which introduced the concept of a polytope complex and its $K$-theory.  The goal of that paper was to define $K$-theory in such a way that on $K_0$ we had the scissors congruence group of the polytope complex.  More concretely, a polytope complex $\C$ is a small double category, which vertically has a Grothendieck topology, and horizontally is a groupoid.  Given a polytope complex $\C$, we can produce a Waldhausen category $\SC(\C)$ such that $K_0(\SC(C))$ is the free abelian group generated by objects of $\C$ under the two relations $[A] = [B]$ if $A$ is horizontally isomorphic to $B$, and $A = \sum_{i=1}^n [A_i]$ if the $A_i$'s are disjoint (have the vertically initial object as their product) and $\{\bd A_i&\rSub&A\ed\}_{i=1}^n$ is a covering family of $A$ in the vertical topology in $\C$.

However, this definition of $K$-theory is problematic from a computational standpoint, as it relies on a Waldhausen category which does not come from an exact category (in the sense of Quillen, \cite{quillen73}), does not have a cylinder functor (as in \cite{waldhausen83}, section 1.6) and which is not good (in the sense of To\"en, \cite{toenvezzosi04}).  This means that very few computational techniques are directly available for analyzing this problem, as most approaches covered in the literature depend on one of these properties.

In this paper we produce several computational results for the $K$-theory of polytope complexes.  The majority of this paper is an analysis of Waldhausen's $S_\dot$-construction in the particular case of the $K$-theory of a polytope complex.  In this case it turns out that for a polytope complex $\C$ we can find a polytope complex $s_n\C$ such that $|wS_n\SC(\C)| \simeq |w\SC(s_n\C)|$.  We can make this construction compatible with the simplicial structure maps from Waldhausen's $S_\dot$-construction, and therefore construct an $S_\dot$-construction directly on the polytope level.

However, as the $S_\dot$-construction adds an extra simplicial dimension, it becomes necessary to be able to define the $K$-theory of a simplicial polytope complex $\C_\dot$. (Here, as well as in the rest of this paper, we consider a simplicial polytope complex to be a simplicial object in the category of polytope complexes; see section \ref{sec:thick} for more details.)  As the definition of $K$-theory relies on geometric realizations, we can define $K(\C_\dot)$ to be the spectrum defined by 
$$K(\C_\dot)_n = | w\underbrace{S_\dot\cdots S_\dot}_n\SC(\C_\dot)|.$$
By analyzing the $S_\dot$-construction on $\SC(\C_\dot)$ we obtain the following computation of the delooping of the $K$-theory of $\C_\dot$:

\begin{theorem}
Let $\C_\dot$ be a simplicial polytope complex.  Let $\sigma\C_\dot$ be the simplicial polytope complex given by the bar construction.  More concretely, we define
$$(\sigma\C_\dot)_n = \underbrace{\C_n\vee \C_n\vee \cdots \vee C_n}_n,$$
with the simplicial structure maps defined in analogously to the usual bar construction.  Then $\Omega K(\sigma \C_\dot) \simeq K(\C_\dot)$.
\end{theorem}

See section \ref{sec:sPoly} and corollary \ref{cor:comp} for more details.

The one computational tool for Waldhausen categories which does not depend in any way on extra assumptions is Waldhausen's cofiber theorem, which, given a functor $G:\E\rightarrow \E'$ between Waldhausen categories constructs a simplicial Waldhausen category $S_\dot G$ whose $K$-theory is the cofiber of the map $K(G):K(\E)\rightarrow K(\E')$.  By passing this computation down through the polytope complex construction of $S_\dot$ we find the following formula for the cofiber of a morphism of simplicial polytope complexes.

\begin{theorem}
Let $g:\C_\dot\rightarrow \D_\dot$ be a morphism of simplicial polytope complexes.  We define a simplicial polytope complex $(\D/g)_\dot$ by setting 
$$(\D/g)_n = \D_n \vee \underbrace{\C_n\vee C_n\vee\cdots \vee C_n}_n.$$
The simplicial structure maps are defined as for $\D_\dot\vee \sigma \C_\dot$, except that $\partial_0$ is induced by the three morphisms
$$\partial_0:\D_n\rightarrow \D_{n-1} \qquad \partial_0g_n:\C_n\rightarrow \D_{n-1} \qquad 1:\C^{\vee n-1} \rightarrow \C^{\vee n-1}.$$
Then 
$$K(\C_\dot) \xrightarrow{K(g)} K(\D_\dot) \rightarrow K((\D/g)_\dot)$$
is a cofiber sequence of spectra.
\end{theorem}

See section \ref{sec:cofiber} and corollary \ref{cor:comp} for more details.

As a consequence of the techniques in section 8 we also get the following proposition:

\begin{proposition} \lbl{prop:XsimeqY}
Let $X$ and $Y$ be homogeneous geodesic $n$-manifolds with a preferred open cover in which the geodesic connecting any two points in a single set is unique.  If there exist preferred open subsets $U\subseteq X$ and $V\subseteq Y$ and an isometry $\varphi:U\rightarrow V$ then the scissors congruence spectra of $X$ and $Y$ are equivalent.
\end{proposition}

The organization of this paper is as follows.  Section 2 covers the notation we use, as well as a basic summary of Waldhausen's $S_\dot$-construction and the results we use about it.  Section 3 defines the category of polytope complexes.  Sections 4 and 5 concern the construction of $s_\dot\C$.  Section 6 describes the fundamental computation necessary for the first theorem, and section 7 the second.  Section 8 wraps up the paper with a basic approximation result about simplicial polytope complexes, which allows us to simplify the formulas computed in sections 6 and 7 to the ones used here.

\section{Preliminaries}

\subsection{Notation}

In this paper, $\C$ and $\D$ will always denote polytope complexes, while $\E$ will be a general Waldhausen category.  We denote the category of polytope complexes and polytope functors by $\PrePolyCpx$.  In a double category (or a polytope complex) we will denote vertical morphism by dotted arrows $\bd A & \rSub & B\ed$ and horizontal morphisms by solid arrows $\bd A & \rTo & B\ed$; functors and morphisms not in a double category or $\SC(\C)$ will be denoted $X\rightarrow Y$.  If a vertical morphism in $\Tw(\C_p)$ is a covering sub-map we will denote it by $\bd A & \rCover & B\ed$.

For any two polytope complexes $\C$ and $\D$, $\C\vee \D$ denotes the polytope complex obtained by identifying the two initial objects.  For a nonnegative integer $n$, we write 
$\C^{\vee n}$ for $\bigvee_{j=1}^n \C$; $\C^{\vee 0}$ will be the trivial polytope complex with no noninitial objects.  

We will often be discussing commutative squares.  Sometimes, in order to save space, we will write a commutative square
\begin{diagram}[small]
A & \rTo & B\\
\dTo^f && \dTo_g \\
C & \rTo & D
\end{diagram}
as
$$f,g:\bd (\bd A & \rTo & B\ed) & \rTo & (\bd C & \rTo & D\ed)\ed.$$

In this paper, whenever we refer to an $n$-simplicial category we will always be referring to a functor $(\Delta^\op)^n \rightarrow \mathbf{Cat}$, rather than an enriched category.  In order to distinguish simplicial objects from non-simplicial objects, we will add a dot as a subscript to a simplicial object; thus $\C$ is a polytope complex, but $\C_\dot$ is a simplicial polytope complex.  For any functor $F$ we will write $F^{(n)}$ for the $n$-fold application of $F$.

\subsection{The $K$-theory of a Waldhausen category} \lbl{sec:S.intro}

This section contains a brief review of Waldhausen's $S_\dot$ construction for $K$-theory, originally introduced in \cite{waldhausen83}, as well as some results which are surely well-known to experts, but for which we could not find a reference.  The proofs of these results are deferred until appendix \ref{app:S.}.


Given a Waldhausen category $\E$, we define $S_n\E$ to be the category of commutative triangles defined as follows.  An object $A$ is a triangle of objects $A_{ij}$ for pairs $1\leq i \leq j \leq n$.  The diagram consists of cofibrations $\bd A_{ij} & \rCofib & A_{(i+1)j}\ed$ and morphisms $\bd A_{ij} & \rTo & A_{i(j+1)}\ed$ such that for every pair $i<j$ and any $1\leq k \leq n-i$ the induced diagram
$$
A_{ij} \longhookrightarrow  A_{(i+k)j} \longrightarrow A_{(i+k)(j+k)}
$$
is a cofiber sequence.  A morphism $\varphi:A\rightarrow B$ consists of morphisms $\varphi_{ij}:A_{ij} \rightarrow B_{ij}$ making the induced diagram commute.  Note that $S_0\E$ is the trivial category with one object and one morphism, and $S_1\E = \E$.  

The $S_n\E$'s assemble into a simplicial object in categories by letting the $k$-th face map remove all objects $A_{ij}$ with $i=k$ or $j=k+1$, and the $k$-th degeneracy repeat a row and column appropriately.  We can assemble the $S_n\E$'s into a simplicial Waldhausen category in the following manner.  A morphism $\varphi:A\rightarrow B\in S_n\E$ is a weak equivalence if $\varphi_{ij}$ is a weak equivalence for all $i<j$.  $\varphi$ is a cofibration if for all $i<j$ the induced morphism 
$$B_{ij} \cup_{A_{ij}} A_{(i+1)j} \longrightarrow B_{(i+1)j}$$
is a cofibration in $\E$.  Note that this means that in particular for all $i<j$ the morphism $\varphi_{ij}$ is a cofibration in $\E$.

We obtain the $K$-theory spectrum of a Waldhausen category $\E$ by defining 
$$K(E)_n = \Omega \left|wS_\dot^{(n)} \E\right|.$$
From proposition 1.5.3 in \cite{waldhausen83} we know that above level $0$ this will be an $\Omega$-spectrum.

We now turn our attention to some tools for computing with Waldhausen categories.  An exact functor of Waldhausen categories $F:\E\rightarrow \E'$, naturally yields a functor between $S_\dot$ constructions, and therefore between the $K$-theory spectra.  We are interested in several cases of such functors which produce equivalences on the $K$-theory level.

The first two examples we consider will be simply inclusions of subcategories.  While a Waldhausen category can contain a lot of morphisms which are neither cofibrations nor weak equivalences, most of these are not important.  We will say that a Waldhausen subcategory $\tilde \E$ of a Waldhausen category $\E$ is a \textsl{simplification of $\E$} if it contains all objects, weak equivalences, and cofibrations of $\E$.  As the $S_\dot$ construction only really looks at these morphisms, it is clear that the inclusion $\tilde \E \rightarrow \E$ induces the identity map $K(\tilde \E)\rightarrow K(\E)$.

Now suppose that $\hat \E$ is a subcategory of $\E$ with the property that any morphism $f\in \E$ can be factored as $hg$, with $h$ an isomorphism and $g\in \hat \E$, and such that $\hat \E$ contains the zero object of $\E$.  Then $\hat \E$ is a Waldhausen category.  Let $\hat S_n\E$ be the full subcategory of $S_n\E$ containing all objects in $S_n\hat E$.  Then  $\hat S_n\E$ is an equivalent Waldhausen subcategory of $S_n\E$, and thus that for all $n\geq 1$, 
$$\left|wS_\dot^{(n-1)}\hat S_\dot \E\right|\simeq \left|wS_\dot^{(n)} \E\right|.$$
Thus we can compute the $K$-theory of $\E$ using only morphisms from $\hat \E$ in the first level of the $S_\dot$ construction.  (For more details, see lemma \ref{lem:subobjects}.)

Now we consider pairs of adjoint functors between Waldhausen categories.  Suppose that we have an adjoint pair of exact functors $F:\E \rightleftarrows \E':G$; these produce a pair of maps $K(F): K(\E) \rightleftarrows K(\E'):K(G)$.  Generally an adjoint pair of functors produces a homotopy equivalence on the classifying space level, so naively we might expect these to be inverse homotopy equivalences.  Unfortunately, in the $S_\dot$ construction we always restrict our attention to weak equivalences in the category, so we need more information than just an adjoint pair of exact functors.  If both the unit and counit of our adjunction is a weak equivalence then we are fine, however, as the adjunction must also restrict to an adjunction on the subcategories of weak equivalences.  We call an adjoint pair of exact functors satisfying this extra condition an \textsl{exact adjoint pair}, and we say that $F$ is \textsl{exactly left adjoint} to $G$.  Given any exact adjoint pair we get a pair of inverse equivalences on the $K$-theory level.

We finish up this section with a short discussion of a simplification of the $S_\dot$ construction.  $S_n$ can be defined more informally as the category whose objects are all choices of $n-1$ composable cofibrations, together with the choices of all cofibers.  As the cofiber of a cofibration $\bd A & \rCofib & B\ed\in \E$ is a pushout, any object $A\in S_n\E$ is defined, up to isomorphism, by the diagram
\begin{diagram}
A_{11} & \rCofib & A_{21} & \rCofib & \cdots & \rCofib & A_{n1},
\end{diagram}
and any morphism $\varphi$ by its restriction to this row.  We will denote the category of such objects $F_n\E$.  We can clearly make $F_n\E$ into a Waldhausen category in a way analogous to the way we made $S_n\E$ into a Waldhausen category.  However, these do not assemble easily into a simplicial Waldhausen category, as $\partial_0$, the $0$-th face map, must take cofibers, and this is only defined up to isomorphism.  Thus while $F_n\E$ is easier to work with on each level, $S_n\E$ is often easier to work with when working with the simplicial structure.  (Note that if in $\E$ all cofibrations come with a canonical choice of cofiber then the $F_n$'s automatically assemble into a simplicial Waldhausen category.  This will be exactly the case that we will be considering later in the paper.)

\section{Thickenings} \lbl{sec:thick}

\begin{definition}
Let $\C$ be a polytope complex.  The polytope complex $\C^\id$ is the full subcategory of $\Tw(\C_p)$ containing all objects $\SCob{a}{i}\in \Tw(\C_p)$ such that for all distinct $i,j\in I$ there exists an $a\in \C$ such that $a_i\times_a a_j = \emptyset$.  The topology on $\C^\id_v$ is defined pointwise.  More precisely, let $X = \SCob{x}{i}$, and $X_\alpha = \{x^{(\alpha)}_j\}_{j\in J_\alpha}$.  We say that $\{p^\alpha:\bd X_\alpha & \rSub & X \ed\}_{\alpha\in A}$ is a covering family if for each $i\in I$ the family $\{P^\alpha_j:\bd x_j^{(\alpha)} & \rSub & x_i\ed \}_{j\in (p^\alpha)^{-1}(i), \alpha\in A}$ is a covering family in $\C$.
\end{definition}

It is easy to check that $-^\id$ is in fact a functor $\PrePolyCpx \rightarrow \PrePolyCpx$.  It will turn out that $-^\id$ is a monad on $\PrePolyCpx$, and that $\SC:\PrePolyCpx \rightarrow \WaldCat$ factors through the inclusion $\PrePolyCpx \rightarrow \Kl(-^\id)$ (the Kleisli category of this monad).  This factorization provides us with extra morphisms between polytopes, which will be exactly the morphisms we need later when we start doing calculations with face maps in the $S_\dot$ construction.

We start by considering the monad structure of $-^\id$.  We have a natural inclusion $\eta_\C:\C\rightarrow \C^\id$ which includes $\C$ into $\C^\id$ as the singleton sets; these assemble into a natural transformation $\eta:1\Rightarrow -^\id$.  This transformation is not a natural isomorphism, even through, morally speaking, $\C^\id$ ought to have the same $K$-theory as $\C$ (as it contains objects which are formal sums of objects of $\C$).  It turns out that once we pass to $\WaldCat$ by $\SC$ we can find a natural ``almost inverse'': an exact left adjoint.

\begin{lemma} \lbl{lem:monad}
The functor $-^\id$ is a monad on $\PrePolyCpx$.
\end{lemma}

\begin{proof}
In order to make $-^\id$ into a monad, we need to define a unit and a multiplication.  The unit $\eta:1_\PrePolyCpx \rightarrow (-^\id)$ will be the natural transformation defined on each polytope complex $\C$ by the natural inclusion $\C\rightarrow \C^\id$ given by including $\C$ as the singleton sets.  The multiplication $\mu:(-^\id)^\id \rightarrow (-^\id)$ is given by the functor ${\C^\id}^\id \rightarrow \C^\id$ given on objects by
$$\{\{a^{(i)}_j\}_{j\in J_i}\}_{i\in I} \longmapsto \{a^{(i)}_j\}_{(i,j)\in \coprod_{i\in I}J_i}.$$
It is a simple definition check to see that with these definitions $(-^\id,\eta,\mu)$ is a monad.
\end{proof}

\begin{lemma} \lbl{lem:splitup} There exists a natural transformation
  $\nu:\SC(-^\id) \Rightarrow \SC(-)$ which for every polytope complex $\C$ is exactly left adjoint to $\SC(\eta_\C):\SC(\C)\rightarrow \SC(\C^\id)$.  The counit of this adjunction will be the identity transformation.
\end{lemma}

\begin{proof}
  Fix a polytope complex $\C$, and let $G= \SC(\eta_\C)$.  To show that $G$ has
  a left adjoint it suffices to show that for any $B\in \SC(\C^\id)$,
  $(B\downarrow G)$ has an initial object.  If we write $B = \SCob{B}{j}$, where
  $B_j = \{b^j_k\}_{k\in K_j}$ then the pure covering sub-map
  \begin{diagram}
    \SCob{B}{j} & \lCover & \{\{b^j_k\}\}_{(j,k)\in \coprod_{j\in J}K_j}
  \end{diagram}
  is the desired object; we define $\nu_\C$ to be the adjoint where $\nu_C(B) =
  \{b^j_k\}_{(j,k)\in \coprod_jK_j}$.  Then the unit is objectwise a pure
  covering sub-map --- thus a weak equivalence --- and the counit is the
  identity, as desired.  To see that these assemble into a natural
  transformation, note that $\nu_\C$ ``flattens'' each set of sets by covering
  it with a set of singletons.  By purely set-theoretic observations it is clear
  that this commutes with applying a functor pointwise to each set element, so
  $\nu$ does, indeed, assemble into a natural transformation.

  It remains to show that $\nu_\C$ is exact.  As left adjoints commute with
  colimits and $\SC(\C)$ has all pushouts, $\nu_\C$ preserves all pushouts.  The
  fact that $F$ preserves cofibrations and weak equivalences follows from the
  definition of $F$ and the fact that covering sub-maps in $\C^\id$ are defined
  pointwise.
\end{proof}

Now consider the Kleisli category of this monad, $\Kl(-^\id)$.  We have a
natural inclusion $\iota:\PrePolyCpx \rightarrow \Kl(-^\id)$ which is the identity
on objects, and takes a polytope functor $F:\C \rightarrow \D$ to the
functor $\eta_\D F$.  Informally
speaking, $\Kl(-^\id)$ is the category of sets of polytopes that can be
``added'', in the sense that we can think of a covering sub-map $\bd\SCob{a}{i}
& \rCover & \SCob{b}{j}\ed$ as expressing the relation $\sum_{i\in I}a_i =
\sum_{j\in J} b_j$.  Using the functor given by lemma \ref{lem:splitup} we can extend $\SC$ to a functor on $\Kl(-^\id)$ rather than just on $\PrePolyCpx$.

\begin{lemma}
  The functor $\SC:\PrePolyCpx \rightarrow \WaldCat$ factors through $\iota$.
\end{lemma}

\begin{proof}
  We define a functor $\tilde \SC:\Kl(-^\id)\rightarrow \WaldCat$ by setting $\tilde\SC(\C) = \SC(\C)$ on polytope complexes $\C\in\Kl(-^\id)$, and by 
  $$\tilde\SC(F:\C\longrightarrow \D) = \nu_\D\SC(F):\SC(\C)\longrightarrow \SC(\D^\id) \longrightarrow\SC(\D).$$
  Note that given any polytope functor $F:\C \rightarrow \D$, 
  $$\tilde\SC(\iota(F)) = \nu_\D\SC(\eta_\D)\SC(F) = \SC(F),$$ as $\nu_\D$ is left adjoint to $\SC(\eta_\D)$ and the counit of the adjunction is the identity.  Thus $\tilde\SC\iota = \SC$, as desired.
\end{proof}

We define the category of polytope complexes, $\PolyCpx$, to be $\Kl(-^\id)$.
By an abuse of notation we will therefore consider $\SC$ to be a functor
$\PolyCpx\rightarrow \WaldCat$.

We finish up with an example of a polytope complex which is an algebra over
$-^\id$, and a polytope complex which is not an algebra over $-^\id$.  Let $\C$
be the polytope complex of nondegenerate polytopes in $\R^n$ with the Euclidean group acting
on it.  We can define a functor $\C^\id\rightarrow \C$ by mapping any set of
pairwise disjoint polytopes to the union of that set (which is well-defined if we define a polytope to be a nonempty union of simplices).  It is easy to check that
this does, in fact, make $\C$ into an algebra over $-^\id$.

Now let $\C$ be the polytope complex of rectangles in $\R^2$ whose sides are
parallel to the coordinate axes, with the group of translations acting on it.
We claim that this is \textsl{not} an algebra over $-^\id$.  Indeed, suppose
that it were, so we have a functor $F:\C^\id\rightarrow \C$.  Consider a
rectangle $R$ split into four sub-rectangles:
\begin{center}
\begin{tabular}{|c|c|}
\hline
$R_1$ & $R_2$ \\
\hline
$R_3$ & \phantom{$\displaystyle{\int}$} $R_4$ \phantom{$\displaystyle{\int}$} \\
\hline
\end{tabular}
\end{center}
We know that $F(\{R\})=R$ and $F(\{R_i\}) = R_i$.  Now consider $F(\{R_1,R_4\})$.  This must sit inside $R$, and also contain both $R_1$ and $R_4$, so it must be $R$.  Similarly, $F(\{R_2,R_3\}) = R$.  But then
$$R = F(\{R_1,R_4\}) \times_{F(\{R\})} F(\{R_2,R_3\}) = F(\{R_1,R_4\}\times_{\{R\}} \{R_2,R_3\}) = F(\emptyset) = \emptyset.$$
Contradiction.  So $\C$ is not an algebra over $-^\id$.

\section{Filtered Polytopes}

The $S_\dot$ construction considers sequences of objects included into one another.  In this section we will look at filtered objects where all of the cofibrations are actually acyclic cofibrations.

Let $W_n\SC(\C)$ be the full subcategory of $F_n\SC(\C)$ which contains all objects
\begin{diagram}
A_1 & \rAcycCofib & A_2 &     \rAcycCofib & \cdots & \rAcycCofib & A_n.
\end{diagram}
We can make $W_n\SC(\C)$ into a Waldhausen category by taking the structure
induced from $F_n\SC(\C)$.  Then $W_n\SC(\C)$ contains $\tilde W_n\SC(\C)$ --- the full subcategory of
$W_n\SC(\C)$ of all such objects which can be represented by only pure sub-maps
--- as an equivalent subcategory (by lemma \ref{lem:subobjects}).

Our goal for this section is to define a polytope complex $f_n\C$ such that $\SC(f_n\C)$ is equivalent (as a Waldhausen category) to $W_n\SC(\C)$.

\begin{definition}
Let $f_n\C$ be the following polytope complex.  An object $A\in f_n\C$ is a diagram
\begin{diagram}
A_1 & \lCover & A_2 & \lCover &\cdots & \lCover & A_n
\end{diagram}
in $\Tw(\C_p)$ such that each $A_i\in \C^\id$ and $A_1$ is a singleton set.  The vertical morphisms $p:\bd A & \rSub & B\ed$ are diagrams
\begin{diagram}[small]
A_1 & \lCover & A_2 & \lCover & \cdots & \lCover & A_n\\
\dSub_{p_1} && \dSub_{p_2} &&&&\dSub_{p_n} \\
B_1 & \lCover & B_2 &\lCover &\cdots & \lCover & B_n
\end{diagram}
in $\C^\id$, and the horizontal morphisms are defined analogously.  We put a topology on $f_n\C$ by defining a family $\{\bd X_\alpha & \rSub & X\ed\}_{\alpha\in A}$ to be a covering family if for each $i=1,\ldots, n$ the family $\{\bd X_{\alpha i} & \rSub & X_i\ed\}_{\alpha\in A}$ is a covering family in $\C^\id$.
\end{definition}

Now we construct the functors which give an isomorphism between $\tilde W_n\SC(\C)$ and $\SC(f_n\C)$.  The functor $H:\SC(f_n\C)\rightarrow \tilde W_n\SC(\C)$ simply takes an object of $\SC(f_n\C)$ to the sequence of its levelwise unions.  More formally, given an object $\SCob{a}{i}$ in $\SC(f_n\C)$, where for each $i\in I$ we have
$$a_i = \bd a_i^1 & \lCover & a_i^2 & \lCover & \cdots & \lCover & a_i^n\ed,$$
with $a_i^j\in \C^\id$, we define an object $H(\SCob{a}{i})\in\tilde W_n\SC(\C)$ by
\begin{diagram}
A_1 & \rAcycCofib & A_2 & \rAcycCofib & \cdots & \rAcycCofib & A_n
\end{diagram}
where $A_j = \coprod_{i\in I} a_i^j \in \Tw(\C_p)$.  In other words, we consider each object $a_i$ to be a diagram in $\Tw(\C_p)$ and we take the coproduct of all of these diagrams.

To construct an inverse $G:\tilde W_n\SC(\C)\rightarrow \SC(f_n\C)$ to this functor we take a diagram in $\tilde W_n\SC(\C)$ and turn it into a coproduct of pure covering sub-maps in $\Tw(\C_p)$.  It will turn out that each of these diagrams represents an object of $\SC(f_n\C)$, which will give us the desired functor.  Given an object $A\in \tilde W_n\SC(\C)$ represented by
\begin{diagram}
A_1 & \rAcycCofib & A_2 & \rAcycCofib & \cdots & \rAcycCofib & A_n
\end{diagram}
we know that we can write every acyclic cofibration in this diagram as a pure covering sub-map.  When a morphism can be represented in this way the representation is unique, so we can in fact consider this object to be a diagram
\begin{diagram}
A_1 & \lCover & A_2 & \lCover & \cdots & \lCover & A_n
\end{diagram}
in $\Tw(\C_p)$.  This sits above an analogous diagram in $\FinSet$.  Given any such diagram in $\FinSet$ we can write it as a coproduct of fibers over the indexing set $I$ of $A_1$. Consequently we can write $A$ as
$$\coprod_{i\in I}\left(\bd A^i_1 & \lCover & A^i_2 & \lCover & \cdots & \lCover & A^i_n\ed\right).$$
We will show that each of these component diagrams actually represents an object of $f_n\C$.  Indeed, we know by definition that $A^i_1$ is a singleton set $\{a_i\}$.  Thus if we write $A^i_j$ as $\SCob{b}{k}$, from the fact that each of the morphisms in the diagram is a sub-map we know that for $K,k'\in K$, $b_k\times_{a_i}b_{k'} = \initial$, so each $A^i_j$ is an object of $\C^\id$.  Thus this diagram is an object of $f_n\C$ as desired.  This definition extends directly to the morphisms as well.

We need to prove that these functors are exact and inverses.  It is easy to see that they are inverses on objects, so we focus our attention on the morphisms in the categories.  To this end we define two projection functors $\pi_1:W_n\SC(\C)\rightarrow \SC(\C)$ and $P_1:f_n\C\rightarrow \C$ which will help us analyze the situation.

\begin{lemma} \lbl{lem:proj1}
  Let $P_1:f_n\C\rightarrow \C$ take a diagram
  $\bd A_1 & \lCover & \cdots & \lCover & A_n\ed$ to the unique element of
  $A_1$.  Then the functor $\SC(P_1)$ is faithful.
\end{lemma}

\begin{proof}
  It suffices to show that given any diagram
  \begin{diagram}[small]
    A_1 & \rAcycCofib & A_2 \\
    \dTo^f && \\
    B_1 & \rAcycCofib & B_2
  \end{diagram}
  there exists at most one morphism $g:\bd A_2 & \rTo & B_2\ed$ that makes the
  diagram commute.  In particular, if we consider the diagram in
  $\Tw((f_n\C)_p)$ representing such a commutative square, we have
  \begin{diagram}[small]
    A_2 & \lSub & A_2' & \rTo & B_2 \\
    \dCover & & \dCover && \dCover \\
    A_1 & \lSub & A_1' & \rTo^\sigma & B_1
  \end{diagram}
  where the morphism $\bd A_2' & \rCover & A_1'\ed$ is a covering sub-map
  because the square commutes.  In particular, this means that $A_2' =
  \sigma^*B_2$.  Thus we can complete the square exactly when we have a sub-map
  $\bd \sigma^*B_2 & \rSub & A_2 \ed$ which makes the left-hand square commute,
  of which there is at most one.
\end{proof}

And, completely analogously, we can prove a symmetric statement about $\pi_1$.

\begin{lemma}
Let $\pi_1:\tilde W_n\SC(\C)\rightarrow \SC(\C)$ be the exact functor which takes an object $\bd A_1 & \rAcycCofib & \cdots & \rAcycCofib & A_n\ed$ to $A_1$.  Then $\pi_1$ is faithful.
\end{lemma}

We can now prove the main result of this section.



\begin{proposition} \lbl{prop:Wneqfn}
$W_n\SC(\C)$ is exactly equivalent to $\SC(f_n\C)$.
\end{proposition}

\begin{proof}
We will show that $G$ and $H$ induce isomorphisms between $\tilde W_n\SC(\C)$ and $\SC(f_n\C)$, which will show the result as $\tilde W_n\SC(\C)$ is exactly equivalent to $W_n\SC(\C)$.

It is clear that $GH$ and $HG$ are the identity on objects, so it remains to show that they are inverses on morphisms.  From the definitions it is easy to see that $\SC(P_1)G = \pi_1$ and that $\pi_1H = \SC(P_1)$, so that
$$\SC(P_1)GH = \pi_1H = \SC(P_1) \qquad \hbox{and} \qquad \pi_1HG = \SC(P_1)G = \pi_1.$$
As $\SC(P_1)$ and $\pi_1$ are both faithful, if we consider these on hom-sets we see that $G$ and $H$ are mutual inverses on any hom-set.  Thus $\tilde W_n\SC(\C)$ is isomorphic to $\SC(f_n\C)$.

It remains to show that $G$ and $H$ are exact functors.  We already know that they preserve pushouts, so all it remains to show is that they preserve cofibrations and weak equivalences.  Note that we know by definition that $\pi_1$ and $\SC(P_1)$ are exact functors; thus in order to show that $G$ and $H$ are exact it suffices to show that $\pi_1$ and $\SC(P_1)$ reflect cofibrations and weak equivalences.

For both of these cases it suffices to show that in $\Tw(\C_p)$ if
\begin{diagram}[small]
A_1 & \lCover^p & A_1' & \rTo^\sigma & B_1 \\
\uCover^i && && \uCover_j \\
A_2 & \lSub^q & \sigma^*B_2 & \rTo^{\tilde \sigma} & B_2
\end{diagram}
commutes and $\sigma$ has an injective set-map, then $q$ is a covering sub-map and $\tilde \sigma$ has an injective set-map.  The first of these is true because $q$ is the pullback along $i$ of $jp$, which is a covering sub-map; the second of these is true because pullbacks preserve injectivity of set-maps.  So we are done.
\end{proof}

\begin{remark}
If we define $P_n$ and $\pi_n$ analogously to $P_1$ and $\pi_1$ we see that $\SC(P_n)$ and $\pi_n$ are exact equivalences of categories.  Thus $\SC(f_n\C)$ and $W_n\SC(\C)$ are exactly equivalent, as they are both equivalent to $\SC(\C)$.  We do not use these functors because they are not compatible with the simplicial maps of $S_n\SC(\C)$, and thus will not give inverse equivalences on the $K$-theory.
\end{remark}

\section{Combing} \lbl{sec:combing}

Let $f:\bd A&\rCofib & B\ed\in \SC(\C)$ be a cofibration.  We define the
\textsl{image} of $f$ to be the cofiber of the canonical cofibration $\bd B/A &
\rCofib & B \ed$ (see \cite{zakharevich10}, corollary 6.8).  We will write the image
of $f$ as $\im(f)$; when the cofibration is clear from context we will often
write is as $\im_B(A)$.  Note that we have an acyclic cofibration
\begin{diagram}
A & \rAcycCofib & \im_B(A)
\end{diagram}
More concretely, if we write $A = \SCob{a}{i}$ and $B = \SCob{b}{j}$, and if $f$ can be represented by covering sub-map $p$ and the shuffle $\sigma$, $\im_B(A) = \{b_j\}_{j\in \im\,\sigma}$.

Now suppose that we are given an object $A=(\bd A_1 & \rCofib & A_2 & \rCofib & \cdots & \rCofib & A_n\ed) \in F_n\SC(\C)$.  Then we define the $i$-th \textsl{strand} of $A$, $\St_i(A)$ to be the diagram
\begin{diagram}
A_i/A_{i-1} & \rAcycCofib & \im_{A_{i+1}}(A_i/A_{i-1}) & \rAcycCofib & \cdots & \rAcycCofib & \im_{A_n}(A_i/A_{i-1}).
\end{diagram}
We can consider $\St_i(A)$ to be an object of $F_n\SC(\C)$ by padding the front with sufficiently many copies of the zero object; then we can canonically write $A = \coprod_{i=1}^n \St_i(A)$.

\begin{definition}
We will say that a morphism $f:A\rightarrow B\in F_n\SC(\C)$ is \textsl{layered} if for all $1\leq i < k \leq n$ the diagram
\begin{diagram}[small]
A_k/A_i & \rCofib & A_k \\
\dTo^{f_k/f_i} & & \dTo_{f_k} \\
B_k/B_i & \rCofib & B_k
\end{diagram}
commutes.  We define $L_n\SC(\C)$ to be the subcategory of $F_n\SC(\C)$ containing all layered morphisms.
\end{definition}

Not all morphisms are layered.  For example, let $X$ be a nonzero object, and let $g:\bd X & \rCofib & Y\ed$ be any cofibration in $\SC(\C)$.  Then $\bd\emptyset & \rCofib & Y\ed$ and $\bd X & \rCofib & Y\ed$ are both objects of $F_2\SC(\C)$ and we have a non-layered morphism
\begin{diagram}[small]
\emptyset & \rCofib & Y \\
\dCofib & & \dEqual \\
X & \rCofib & Y
\end{diagram}
between them.  As all cofibers of acyclic cofibrations are trivial, all morphisms of $W_n\SC(\C)$ are layered.  In fact, if we let $I_{ni}:F_{n-i+1}\SC(\C) \rightarrow F_n\SC(\C)$ be the functor which pads a diagram with $i$ copies of $\emptyset$ at the beginning, then the restriction of $I_{ni}$ to $L_{n-i+1}\SC(\C)$ has its image in $L_n\SC(\C)$. 

\begin{lemma} \lbl{lem:layering} $\hbox{ }$
\begin{enumerate}
\item $f$ is layered if and only if for all $1\leq i < n$, the morphism
$$f_{i,i+1}:(\bd A_i & \rCofib & A_{i+1}\ed) \longrightarrow (\bd B_i & \rCofib & B_{i+1}\ed)\in F_2\SC(\C)$$
is layered.
\item  Given any commutative square $(\bd A_1 & \rCofib & A_2\ed)\rightarrow (\bd B_1&\rCofib& B_2\ed)$ we have an induced commutative square $(\bd \im_{A_2}(A_1) &\rCofib & A_2\ed) \rightarrow (\bd \im_{B_2}(B_1) & \rCofib & B_2\ed)$.  Thus if $(\bd A_1 & \rCofib & A_2 \ed) \longrightarrow (\bd B_1 & \rCofib & B_2\ed)$ is layered then so is $(\bd A_2/A_1 & \rCofib & A_2 \ed) \longrightarrow (\bd B_2/B_1 & \rCofib & B_2\ed)$.
\end{enumerate}
\end{lemma}

\begin{proof} $\hbox{ }$
\begin{enumerate}
\item The forwards direction is trivial, so it suffices to prove the backwards direction.  We will prove it by induction on $k$.  For $k=i+1$ this is given.  Now suppose that it is true up to $k$.  Then we have the following diagram
\begin{diagram}[small]
&&A_k &&\lCofib && A_k/A_i \\
&&\dLine & \rdCofib &&& \vLine & \rdCofib \\
A_{k+1}/A_k &  \rCofib   &\HonV&& A_{k+1} && \HonV & \lCofib & A_{k+1} / A_i \\
&&\dTo&&\dTo && \dTo && \dTo \\
\dTo &&B_k & \lTo & \VonH & \lHookline & B_k/B_i \\
&&& \rdCofib & & & & \rdCofib \\
B_{k+1}/B_k&&\rCofib&& B_{k+1} & & \lCofib && B_{k+1}/B_i
\end{diagram}
in which we know that every face other than the front one commutes; we want to show that the front face also commutes.  Let 
\[\alpha:\bd A_{k+1}/A_i & \rCofib & A_{k+1} & \rTo & B_{k+1}\ed \qquad \beta:\bd A_{k+1}/A_i & \rTo & B_{k+1}/B_i & \rCofib & B_{k+1}\ed;\]
we want to show that $\alpha = \beta$.  As $A_{k+1}/A_i$ is the pushout of the diagram
\begin{diagram}
A_{k+1} & \lCofib & A_k & \rTo & A_k/A_i,
\end{diagram}
it suffices to show that $f\alpha=f\beta$ and $g\alpha = g\beta$ for $f:\bd A_k/A_i & \rCofib & A_{k+1}/A_o\ed$ and $g:\bd A_{k+1} & \rTo & A_{k+1}/A_i\ed$.  The first of these follows directly from the fact that all faces of the cube but the front one commute.  For the second of these, note that we have a weak equivalence $\bd A_k \amalg A_{k+1}/A_k & \rWeakEquiv A_{k+1}\ed$ and weak equivalences are epimorphisms, so in fact it suffices to show that $g_1\alpha = g_1\beta$ and $g_2\alpha = g_2\beta$ for
\[g_1 = \bd A_k & \rCofib & A_{k+1} & \rTo & A_{k+1}/A_i\ed \qquad \hbox{and} \qquad g_2 = \bd A_{k+1}/A_k & \rCofib & A_{k+1} & \rTo & A_{k+1}/A_i\ed.\]
The first of these follows from a simple diagram chase, keeping in mind that all horizontal cofibrations in this cube are actually sections of cofiber maps.  The second of these also turns into a simple diagram chase after noting that for any sequence of cofibrations $\bd X & \rCofib & Y & \rCofib & Z\ed$ in $\SC(\C)$ we have
$$\bd Z/Y & \rCofib & Z & \rTo & Z/X & \rCofib Z \ed = \bd Z/Y & \rCofib & Z\ed.$$

\item  Note that if we have a commutative square $(\bd A_1 & \rCofib & A_2\ed)\rightarrow (\bd B_1 &\rCofib & B_2\ed)$ it can be represented by the following commutative diagram in $\Tw(\C_p)$:
\begin{diagram}[small]
A_1 & \lCover & A_1' & \rTo & A_2 \\
\uSub && \uSub &\star& \uSub \\
\tilde A_1 & \lCover & X & \rTo & \tilde A_2 \\
\dTo &\star & \dTo & & \dTo \\
B_1 & \lCover & B_1' & \rTo & B_2
\end{diagram}
where the starred squares are pullbacks.  We know $A_1' \cong \im_{A_2}(A_1)$ and $B_1' \cong \im_{B_2}(B_1)$ and the middle column in the diagram represents a morphism between them.  In fact, the right-hand half of this diagram is --- up to isomorphism --- exactly the square that the lemma states exists.  The second part of the statement follows because the cofiber of $\bd A_2/A_1 & \rCofib & A_2\ed$ is exactly $\im_{A_2}(A_1)$.
\end{enumerate}
\end{proof}

\begin{lemma} \lbl{lem:Lnstructure}
$L_n\SC(\C)$ is a Waldhausen category which is a simplification of $F_n\SC(\C)$.  The cofibrations (resp. weak equivalences) in $L_n\SC(\C)$ are exactly the morphisms which are levelwise cofibrations (resp. weak equivalences).
\end{lemma}

We postpone the proof of this lemma until appendix \ref{app:technical} as it is technical and not particularly illuminating.

\begin{lemma} \lbl{lem:Sti}
$\St_i$ is an exact functor $L_n\SC(\C)\rightarrow W_{n-i+1}\SC(\C)$.  We have a natural transformation $\eta_i:I_{ni}\St_i \rightarrow \mathrm{id}$ given by the natural inclusions $\bd \im_{A_k}(A_i/A_{i-1}) & \rCofib & A_k\ed$.    On $W_{n-i+1}\SC(\C)$,
$$\St_iI_{ni} = \mathrm{id} \qquad \hbox{and} \qquad \St_j I_{ni} \St_i = 0$$
for $i\neq j$.  
\end{lemma}

\begin{proof}
Let $f:A\rightarrow B \in L_n\SC(\C)$.  We claim that the morphism 
\begin{diagram}[small]
A_i/A_{i-1} & \rCofib & A_{i+1} & \rCofib & \cdots & \rCofib & A_n\\
\dTo^{f_i/f_{i-1}} && \dTo_{f_{i+1}} && && \dTo_{f_n}\\
B_i/B_{i-1} & \rCofib & B_{i+1} & \rCofib & \cdots & \rCofib & B_n
\end{diagram}
in $F_{n-i+1}\SC(\C)$ is also layered.  By lemma \ref{lem:layering}(1) we know that it suffices to check that each square in this diagram satisfies the layering condition.  All squares but the first one satisfy it because $f$ is layered.  The first square can be factored as 
\begin{diagram}[small]
A_i/A_{i-1} & \rCofib & A_i & \rCofib & A_{i+1} \\
\dTo^{f_i/f_{i-1}} && \dTo_{f_i} && \dTo_{f_{i+1}} \\
B_i/B_{i-1} & \rCofib & B_i & \rCofib& B_{i+1}
\end{diagram}
The right-hand square satisfies the layering condition because $f$ is layered; the left-hand square satisfies it by lemma \ref{lem:layering}(2).  If we let $T_i:L_n\SC(\C)\rightarrow L_{n-i+1}\SC(\C)$ be the functor taking an object to this truncation then $T_i$ is exact, as by lemma \ref{lem:Lnstructure} layered cofibrations are exactly levelwise.  Note that $T_i I_{ni} = \mathrm{id}$ and we have a natural transformation $\eta':I_{ni} T_i \rightarrow \mathrm{id}$.

We can write $\St_i = \St_1T_i$; thus if we can prove the lemma for $i=1$ we will be done.  The fact that $f$ is layered implies that $\St_1$ is a functor $L_n\SC(\C) \rightarrow W_n\SC(\C)$ (as it is obtained by taking levelwise cofibers in a commutative diagram).  As colimits commute past one another, we see that this preserves pushouts along cofibrations.  Thus to see that $\St_1$ is exact it remains to show that it preserves cofibrations and weak equivalences, which is true because both weak equivalences and cofibrations are preserved by taking cofibers, and $\St_1$ simply takes two successive cofibers.

The natural transformation $\eta_1$ is obtained by factoring each cofibration $\bd A_1&\rCofib & A_k\ed$ through the weak equivalence $\bd A_1 & \rAcycCofib & \im_{A_k}(A_1)\ed$.  By the discussion in the proof of lemma \ref{lem:layering}(2) this will in fact be a natural transformation.

Now we show the last part of the lemma.  It is a simple computation to see that $\St_i|_{W_m\SC(\C)}$ is the identity if $i=1$, and $0$ otherwise.  Thus $\St_iI_{ni} = \St_1T_iI_{ni} = \St_1$ is the identity.  If $j<i$ then the $j$-th component of $I_{ni}\St_i$ is $\initial$, so $\St_j I_{ni}\St_i = 0$ trivially.  If $j>i$ then $\St_j I_{ni} \St_i = \St_{j-i+1}T_iI_{ni}\St_i = \St_{j-i+1} \St_i = 0$ because $j-i+1>1$.  Thus we are done.
\end{proof}

\begin{proposition} \lbl{prop:Lnstrands}
Let $CP:\prod_{m=1}^n W_m\SC(\C)\rightarrow L_n\SC(\C)$ be the functor which takes an $n$-tuple $(X_1.\ldots,X_n)$ to $\coprod_{i=1}^n I_{n(n-i+1)}(X_i)$.  We have an exact equivalence of categories
$$\St:L_n\SC(\C) \rightleftarrows \prod_{m=1}^n W_m\SC(\C) : CP,$$
where $\St$ is induced by the functors $\St_m$ for $m=1,\ldots,n$.
\end{proposition}

\begin{proof}
We first show that these form an equivalence of categories.  From lemma \ref{lem:Sti} above, we know that the composition $\St\circ CP$ is the identity on each component (as $\St_i\St_jA$ is the zero object for $i\neq j$), and thus the identity functor.  On the other hand, the composition $CP\circ \St$ has a natural transformation $\eta=\coprod_{m=1}^n \eta_{n-m+1}:CP\circ \St\rightarrow \mathrm{id}$; it remains to show that $\eta$ is in fact a natural isomorphism.  However, for every object $A$, $\eta_A$ is simply the natural morphism $\coprod_{i=1}^n \St_i(A) \rightarrow A$, which is clearly an isomorphism.  So these are in fact inverse equivalences.

As each component of $CP$ is exact (as cofibrations and weak equivalences in $L_n\SC(\C)$ are levelwise) we know that $CP$ is exact.  On the other hand, $\St_i$ is exact for all $i$, so $\St$ is exact.  So we are done.
\end{proof}

The functor $\St$ ``combs'' an object of $L_n\SC(\C)$ by separating all of the strands of different lengths.

\section{Simplicial Polytope Complexes} \lbl{sec:sPoly}

Our goal for this section is to assemble the $f_i\C$ into a simplicial polytope complex which will mimic Waldhausen's $S_\dot$ construction.

Given $1\leq i \leq n$ we define a morphism $\partial_i^{(n)}:f_n\C\rightarrow f_{n-1}\C$ in $\PolyCpx$ induced by skipping the $i$-th term.  If $i>1$ this functor comes from $\PrePolyCpx$; if $i=1$ then we cut off the singleton element from the front, and therefore have to split the rest of the object into fibers over the different polytopes in the (newly) first set.  (This is why $\partial_i^{(n)}$ is a morphism in $\PolyCpx$ rather than in $\PrePolyCpx$.)  We define the morphism $\sigma_i^{(n)}:f_n\C\rightarrow f_{n+1}\C$ to be the morphism of $\PolyCpx$ given by the polytope functor which repeats the $i$-th stage.  For $i\leq 0$ we define the morphisms $\sigma_i^{(n)}:f_n\C\rightarrow f_n\C$ and $\partial_i^{(n)}:f_n\C\rightarrow f_n\C$ to be the identity on $f_n\C$.  Note that the only one of these morphisms that does not come from $\PrePolyCpx$ is $\partial^{(n)}_1$.

\begin{definition}
Let $s_n\C = \bigvee_{i=1}^n f_i\C$.  We define simplicial structure maps between these by
$$\partial_0 = (0:f_n\C\rightarrow s_{n-1}\C) \vee\bigvee_{i=1}^{n-1} (1:f_i\C\rightarrow f_i\C),$$
where $0$ is the polytope functor sending everything to the initial object $\initial$,
$$\partial_i = \bigvee_{i=1}^n \partial^{(j)}_{n-j+i} \qquad \hbox{for } i\geq 1,$$ and $$\sigma_i:s_n\C\rightarrow s_{n+1}\C = \bigvee_{j=1}^n \sigma^{(j)}_{n-j+i} \qquad \hbox{for } i\geq 0.$$
\end{definition}

It is easy to see that with the $\partial_i$'s as the face maps and the $\sigma_i$'s as the degeneracy maps,
$s_\dot\C$ becomes a simplicial polytope complex.

Putting proposition \ref{prop:Lnstrands} together with proposition \ref{prop:Wneqfn} we see that we have an exact equivalence $L_n\SC(\C)\rightarrow \prod_{m=1}^n \SC(f_m\C)$.  However, $\prod_{m=1}^n \SC(f_m\C)$ is exactly equivalent to $\SC(\bigvee_{m=1}^n f_m\C) = \SC(s_n\C)$.  Thus we have proved the following:

\begin{corollary} \lbl{cor:F.}
Let $H_m:\SC(f_m\C)\rightarrow \tilde W_m\SC(\C)$ be the functor in proposition \ref{prop:Wneqfn}, $\iota_m:\tilde W_m\SC(\C)\rightarrow W_m\SC(\C)$ be the natural inclusion, and $CP_n$ be the functor from proposition \ref{prop:Lnstrands}.   Then $F_n = CP_n\circ \left(\prod_{m=1}^n \iota_m\circ H_m\right)$ is an exact equivalence of categories.
\end{corollary}

Now we know that $L_\dot\SC(\C)$ is a simplicial Waldhausen category, and $\SC(s_\dot\C)$ is a simplicial Waldhausen category.  $F_\dot$ is a levelwise exact equivalence; we would like to show that it commutes with the simplicial maps, and therefore assembles to a functor of simplicial Waldhausen categories.  This will allow us to conclude that the two constructions give equivalent $K$-theory spectra, and thus that we can work directly with the $\SC(s_\dot\C)$ definition.

\begin{proposition} \lbl{prop:sWaldfunc}
The functor $F_\dot:\SC(s_\dot\C)\rightarrow L_\dot\SC(\C)$ is an exact equivalence of simplicial Waldhausen categories.
\end{proposition}

\begin{proof}
First we will show that $F_\dot$ is, in fact, a functor of simplicial Waldhausen categories.  In particular, it suffices to show that the following two diagrams commute for each $i$:
\begin{diagram}[small]
\SC(s_n\C) & \rTo^{\sigma_i} & \SC(s_{n+1}\C) \\
\dTo^{F_n} && \dTo_{F_{n+1}} \\
L_n\SC(\C) & \rTo^{\sigma_i} & L_{n+1}\SC(\C)
\end{diagram}
and
\begin{diagram}[small]
\SC(s_n\C) & \rTo^{\partial_i} & \SC((s_{n-1}\C)^\id) & \rTo^{\nu_{s_{n-1}\C}} & \SC(s_{n-1}\C) \\
\dTo^{F_n} &&&& \dTo_{F_{n-1}}\\
L_n\SC(\C) && \rTo^{\partial_i} && L_{n-1}\SC(\C)
\end{diagram}
where the first diagram is a square because all $\sigma_i$'s come from morphisms in $\PrePolyCpx$.  Both of these diagrams commute by simple computations, since $F_n$ takes "levelwise unions".

Now by corollary \ref{cor:F.} we know that levelwise $F_n$ is an exact equivalence of Waldhausen categories.  In addition, propositions \ref{prop:Lnstrands} and \ref{prop:Wneqfn} give us formulas for the levelwise inverse equivalences; an analogous proof shows that these also assemble into a functor of simplicial Waldhausen categories.  Thus $F_\dot$ is an equivalence of simplicial Waldhausen categories, as desired.
\end{proof}

Suppose that $\C_\dot$ is a simplicial polytope complex.   We define the $K$-theory spectrum of $\C_\dot$ by
$$K(\C_\dot)_n = \left|NwS_\dot^{(n)}\SC(\C_\dot):(\Delta^\op)^{n+2} \rightarrow \Sets\right|.$$
(Note that this definition is compatible with the $K$-theory of a polytope complex, if we consider a polytope complex as a constant simplicial complex.)

\begin{lemma} \lbl{lem:spectrum}
$K(\C_\dot)$ is a spectrum, which is an $\Omega$-spectrum above level $0$.
\end{lemma}

In the proof of this lemma we use the following obvious generalization of lemma 5.2 in \cite{waldhausen78}.  A \textsl{fiber sequence} of multisimplicial categories is a sequence which is a fibration sequence up to homotopy after geometric realization of the nerves.

\begin{lemma} (\cite{waldhausen78}, 5.2) \lbl{lem:multfib}
Let 
$$X_{\dot\,\dot\,\dot} \longrightarrow Y_{\dot\,\dot\,\dot} \longrightarrow Z_{\dot\,\dot\,\dot}$$
be a diagram of $n$-simplicial categories.  Suppose that the following three conditions hold:
\begin{itemize}
\item the composite morphism is constant,
\item $Z_{\dot\,\dot\,\dot m}$ is connected for all $m\geq 0$, and
\item $X_{\dot\,\dot\,\dot m} \rightarrow Y_{\dot\,\dot\,\dot m} \rightarrow Z_{\dot\,\dot\,\dot m}$ is a fiber sequence for all $m\geq 0$.
\end{itemize}
Then $X_{\dot\,\dot\,\dot} \rightarrow Y_{\dot\,\dot\,\dot} \rightarrow Z_{\dot\,\dot\,\dot}$ is a fiber sequence.
\end{lemma}

We now prove lemma \ref{lem:spectrum}.

\begin{proof}[Proof of lemma \ref{lem:spectrum}]
Suppose that $X_{\cdots}$ is an $n$-simplicial object; we will write $PX_{\dot\,\dot\,\dot}$ for the $n$-simplicial object in which $PX_{m_1 \cdots m_n} = X_{(m_1+1)m_2\cdots m_n}$.

Consider the following sequence of functors.
$$wS_1S_\dot^{(n-1)}\SC(\C_\dot) \longrightarrow PwS_\dot^{(n)} \SC(\C_\dot) \longrightarrow wS_\dot^{(n)} \SC(\C_\dot),$$
where the first functor is the constant inclusion as the $0$-space, and the second is the contraction induced by $\partial_0$ on the outermost simplicial level.  As $S_0\E$ is constant for any Waldhausen category $\E$, the composite of the diagram is constant.  Similarly, for any $m\geq 0$ $wS_\dot^{(n)} \SC(\C_{m})$ is connected, as if we plug in $0$ to any of the $S_\dot$-directions we get a constant category.    In addition, by proposition 1.5.3 of \cite{waldhausen83}, this is a fiber sequence if $n\geq 2$.  Thus by lemma \ref{lem:multfib} for $n+1$-simplicial categories, the original diagram was a fiber sequence.

As $S_1\E = \E$ for all Waldhausen categories $\E$, this fiber sequence gives us, for every $n\geq 2$, an induced map $K(\C_\dot)_{n-1}\rightarrow \Omega K(\C_\dot)_n$ which is a weak equivalence.  It remains to show that we have a morphism $K(\C_\dot)_0 \rightarrow \Omega K(\C_\dot)_1$.  Considering the above sequence for $n=1$ we have
$$wS_1\SC(\C_\dot) \longrightarrow PwS_\dot\SC(\C_\dot) \longrightarrow wS_\dot\SC(\C_\dot).$$
While the third criterion from lemma \ref{lem:multfib} no longer applies, the composition is still constant and $PwS_\dot\SC(\C_\dot)$ is still contractible, so we have a well-defined (up to homotopy) morphism $K(\C_\dot)_0 \rightarrow K(\C_\dot)_1$, as desired.
\end{proof}

For all $n$ $L_\dot\SC(\C_\dot)$ is a simplification of $S_\dot\SC(\C_\dot)$, so $K(\C_\dot)_n = |NwL_\dot^{(n)}\SC(\C_\dot)|$.  Now let 
$$\tilde K(\C_\dot)_n = \left|Nw\SC(s_\dot^{(n)}\C_\dot):(\Delta^\op)^{n+2}\rightarrow \Sets\right|.$$
(This is clearly a spectrum, as the proof of \ref{lem:spectrum} translates directly to this case.)  By proposition \ref{prop:sWaldfunc} we have a morphism $\tilde K(\C_\dot)\rightarrow \tilde K(\C_\dot)$ induced by $F_\dot$, which is levelwise an equivalence (and thus an equivalence of spectra).  In particular we can take $\tilde K(\C_\dot)$ to be the definition of the $K$-theory of a simplicial polytope complex.

The main advantage of passing to simplicial polytope complexes is that it allows us to start the $S_\dot$-construction at any level, and thus compute deloopings of our $K$-theory spectra on the polytope complex level.

\begin{corollary} \lbl{cor:sigmaA}
Let $\C_\dot$ be a simplicial polytope complex, and let $\sigma \C_\dot$ be the simplicial polytope complex with
$$(\sigma \C_\dot)_k = s_k\C_k.$$
Then $\Omega K(\sigma \C_\dot) \simeq K(\C_\dot)$.
\end{corollary}

\begin{proof}
Geometric realizations on multisimplicial sets simply look at the diagonal, so
$$\tilde K(\C_\dot)_n = \left|[k] \longmapsto Nw\SC({s_k}^{(n)} \C_k)\right|.$$
Thus 
$$\tilde K(\sigma \C_\dot)_n = \left|[k] \longmapsto Nw\SC({s_k}^{(n)} (s_k\C_k))\right| = \left|[k] \longmapsto Nw\SC({s_k}^{(n+1)} \C_k)\right| = \tilde K(\C_\dot)_{n+1}.$$
Thus $\tilde K(\sigma \C_\dot)$ is a spectrum which is a shift of $\tilde K(\C_\dot)$, so $\tilde K(\C_\dot) \simeq \Omega \tilde K(\sigma \C_\dot)$.  As $\tilde K(\C_\dot)\simeq K(\C_\dot)$, the desired result follows.  
\end{proof}


Using this corollary we can compute a polytope complex model of every sphere.  From the examples in \cite{zakharevich10} we know that the polytope complex $\S=\bd \initial & \rSub & *\ed$ has $K(\S)$ equal to the sphere spectrum (up to stable equivalence).  In order to get $S^1$ we need to deloop $\S$.  Note that $f_n\S = \S$ for all $n$, so $s_n\S = \S^{\vee n}$.  So the simplicial polytope complex which gives $S^1$ on $K$-theory is 
$$(\S^{\vee 0}, \S^{\vee 1}, \S^{\vee 2}, \ldots, \S^{\vee n}, \ldots).$$
Since $f_n(\C\vee \D) = f_n\C\vee f_n\D$ we know that $f_n\S^{\vee m} = \S^{\vee m}$, so we compute that the simplicial polytope complex which gives $S^2$ on $K$-theory is
$$(\S^{\vee 0}, \S^{\vee 1}, \S^{\vee 2^2}, \S^{\vee 3^2}, \ldots,\S^{\vee n^2}, \ldots).$$
In general we obtain $S^k$ as the $K$-theory of
$$(\S^{\vee 0^k}, \S^{\vee 1^k}, \S^{\vee 2^k}, \S^{\vee 3^k}, \ldots, \S^{\vee n^k}, \ldots).$$
Note that in fact this works for $k=0$ as well, as long as we interpret $0^0$ to be $1$.

\section{Cofibers} \lbl{sec:cofiber}

Waldhausen's cofiber lemma (see \cite{waldhausen83}, corollary 1.5.7) gives the following formula for the cofiber of a functor
$G:\E\rightarrow \E'$.  We define $S_nG$ to be the pullback of the diagram
$$S_n\E \xrightarrow{\ S_nG\ } S_n\E' \xleftarrow{\ \ \partial_0\ \ } S_{n+1}\E'.$$
Define $K(S_\dot G)$ by
$$K(S_\dot G)_n = \left| wS_\dot^{(n)} S_\dot G \right|.$$
Then the sequence $K(\E) \rightarrow K(\E') \rightarrow K(S_\dot G)$ is a homotopy cofiber sequence.

Our goal for this section is to compute a version of this for polytope complexes.

\begin{definition}
Let $g:\C\rightarrow \D$ be a morphism in $\PolyCpx$.  We define $\D/g$ to be the simplicial polytope complex with $(\D/g)_n = f_{n+1}\D\vee s_n\C$ and the following structure maps.  For all $i>0$, $\partial_i:(\D/g)_n\rightarrow (\D/g)_{n-1}$ is induced by the two morphisms $$\partial^{(n+1)}_{i+1}:f_{n+1}\D\rightarrow f_n\D \qquad \hbox{and} \qquad \partial_i:s_n\C\rightarrow s_{n-1}\C_\dot$$  Similarly, for all $i\geq 0$, $\sigma_i:(\D/g)_n\rightarrow (\D/g)_{n+1}$ is induced by the morphisms $$\sigma^{(n+1)}_{i+1}:f_{n+1}\D\rightarrow f_{n+2}\D \qquad \hbox{and}\qquad \sigma_i:\sigma_n\C\rightarrow \sigma_{n+1}\C_\dot$$  $\partial_0$, on the other hand, is induced by the three morphisms $$\partial_1^{(n+1)}:f_{n+1}\D\rightarrow f_n\D \qquad f_ng:f_n\C\rightarrow f_n\D \qquad 1:s_{n-1}\C\rightarrow s_{n-1}\C_\dot$$
\end{definition}

When $g$ is clear from context we will often write $\D/\C$ instead of $\D/g$.  For every $n\geq 0$ we have a diagram of polytope complexes
$$\D \longrightarrow (\D/g)_n \longrightarrow s_n\C$$
given by the inclusion $\D\rightarrow f_{n+1}\D$ (as the constant objects) and the projection down to $s_n\C$.  Then $\SC(\D/g)$ is the pullback of
$$\SC(s_\dot\C)\xrightarrow{\SC(s_\dot g)} \SC(s_\dot\D) \xleftarrow{\ \ \partial_0\ \ } PSC(s_\dot\D),$$
which exactly mirrors the construction of $S_\dot G$.  (This is clear from an analysis of $S_\dot\SC(g)$ analogous to that of section \ref{sec:combing}.)  In particular we have from \cite{waldhausen83} proposition 1.5.5 and corollary 1.5.7 that
$$wS_\dot^{(n)}\SC(\C) \longrightarrow wS_\dot^{(n)}\SC(\D) \longrightarrow wS_\dot^{(n)}\SC(\D/g) \longrightarrow wS_\dot^{(n)}\SC(s_\dot\C)$$
is a fiber sequence of $n+1$-simplicial categories.	

Generalizing this to simplicial polytope complexes, we have the following proposition.

\begin{proposition} \lbl{prop:cofibers}
Let $g_\dot:\C_\dot\rightarrow \D_\dot$ be a morphism of simplicial polytope complexes, and write $(\D/g)_\dot$ for the simplicial polytope complex where $(\D/g)_n = (\D_n/g_n)_n$.  Then we have a cofiber sequence of spectra
$$K(\C_\dot) \longrightarrow K(\D_\dot) \longrightarrow K((\D/g)_\dot),$$
where the first map is induced by $g_\dot$, and the second is induced for each $n$ by the inclusion $\D_n\rightarrow (\D_n/g_n)_n$ as the constant objects of $f_{n+1}\D_n$.
\end{proposition}

\begin{proof}
As all cofiber sequences in spectra are also fiber sequences, it suffices to show that this is a fiber sequence.  As homotopy pullbacks in spectra are levelwise (see, for example, \cite{kreck05}, section 18.3), it suffices to show that for all $n\geq 0$, $K(\C_\dot)_n\rightarrow K(\D_\dot)_n \rightarrow K((\D/g)_\dot)_n$ is a homotopy fiber sequence.  However, as we know that above level $0$ all of these are $\Omega$-spectra it in fact suffices to show this for $n>0$.

Thus in particular we want to show that for all $n>0$ the sequence
\[wS_\dot^{(n)}\SC(\C_\dot)\longrightarrow wS_\dot^{(n)} \SC(\D_\dot) \longrightarrow wS_\dot^{(n)} \SC((\D/g)_\dot)\]
is a homotopy fiber sequence of $n+1$-simplicial categories.  Let $\D_\dot/g_\dot$ be the bisimplicial polytope complex where the $(k,\ell)$-th polytope complex is $(\D_k/g_k)_\ell$. It will suffice to show that 
\[wS_\dot^{(n)}\SC(\D_\dot) \longrightarrow wS_\dot^{(n)}\SC(\D_\dot/g.) \longrightarrow wS_\dot^{(n)} \SC(s_\dot\C_\dot)\]
is a fiber 	sequence of $n+2$-simplicial categories (where $\D_\dot$ is now considered a bisimplicial polytope complex); in this diagram the second morphism is induced by the projection $(\D_k/g_k)_\ell\rightarrow s_\ell\C_k$ for all pairs $(k,\ell)$.  Then by comparing this sequence for the functor $1:\C_\dot\rightarrow \C_\dot$ to the functor $g_\dot$ we will be able to conclude the desired result.  (In this approach we follow Waldhausen in \cite{waldhausen83}, 1.5.6.)

We show this by applying lemma \ref{lem:multfib}, where we fix the index of the simplicial direction of $\C_\dot$ and $\D_\dot$.  The composition of the two functors is constant, as we first include $\D_\dot$ and then project away from it, and as we do not fix any of the $S_\dot$ indices the last space will be connected.  Thus we want 
\[wS_\dot^{(n)}\SC(\D_k) \longrightarrow wS_\dot^{(n)}\SC(\D_k/g_k)\longrightarrow wS_\dot^{(n)}\SC(s_\dot\C_k)\]
to be a fiber sequence, which holds by our discussion above.  So we are done. 
\end{proof}

\section{Wide and tall subcategories}

We now take a slight detour into a more computational direction.  Consider the case of a polytope complex $\D$, together with a subcomplex $\C$.  We know that the inclusion $\C\rightarrow \D$ induces a map $K(\C)\rightarrow K(\D)$.  The goal of this section is to give sufficient conditions on $\C$ which will ensure that this map is an equivalence.

We start off the section with an easy computational result which will make later proofs much simpler.

\begin{lemma} \lbl{lem:cofiltpre}
For any object $Y\in w\SC(\C)$, $(Y\downarrow w\SC(\C))$ is a cofiltered preorder.
\end{lemma}

\begin{proof}
In order to see that $(Y\downarrow w\SC(\C)$ is a preorder it suffices to show that given any diagram
\begin{diagram}
A & \lWeakEquiv^f & Y & \rWeakEquiv^g & B
\end{diagram}
in $w\SC(\C)$ there exists at most one morphism $\bd A & \rWeakEquiv & B\ed$ that makes the diagram commute.  This diagram is represented by a diagram in $\Tw(\C_p)$
\begin{diagram}[small]
&& Y' && && Y'' \\
A & \ldTo(2,1)^\sigma && \rdCover(2,1)^p & Y & \ldCover(2,1)^q && \rdTo(2,1)^\tau & B
\end{diagram}
where $\sigma$ and $\tau$ are isomorphisms.  Then morphisms $h:A\rightarrow B$ such that $g = hf$ correspond exactly to factorizations of $q$ through $p$; as $(\Tw(\C_p)^\Sub\downarrow T)$ is a preorder, there is at most one of these and we are done. 

Thus it remains to show that this preorder is cofiltered; in particular, we want to find an object below $A$ and $B$ under $Y$.  Given a shuffle $\sigma'$, let $f_{\sigma'}\in\SC(\C)$ be the pure shuffle defined by $\sigma'$; similarly, for a sub-map $p'$ let $f^{p'}\in \SC(\C)$ be the pure sub-map defined by $p'$.  Let $Z = Y'\times_{Y} Y''$ be the vertical pullback of $p$ and $q$.  Then, the pullback of 
\begin{diagram}
A & \rTo^{\sigma^{-1}} & Y' & \lCover & Z
\end{diagram}
gives a weak equivalence $\bd A & \rWeakEquiv & Z\ed$, and analogously we have a weak equivalence $\bd B & \rWeakEquiv & Z\ed$.  As these commute under $Y$ we see that $(Y\downarrow w\SC(\C))$ is cofiltered, as desired.
\end{proof}

The first condition that we need in order to have an equivalence on $K$-theory is that we must have the same $K_0$; more specifically, we need every object of $\SC(\D)$ to be weakly equivalent to something in $\SC(\C)$.  As a condition on polytope complexes, this turns into the following definition.

\begin{definition}
Suppose that $\D$ is a polytope complex and $\C\rightarrow \D$ is an inclusion of polytope complexes.  We say that $\C$ has \textsl{sufficiently many covers} if for every object $B\in \D$ there exists a finite covering family $\{\bd B_\alpha & \rSub & B\ed\}_{\alpha\in A}$ such that the $B_\alpha$ are pairwise disjoint, and such that every $B_\alpha$ is horizontally isomorphic to an object of $\C$.
\end{definition}

Our first approximation result is almost obvious: if we cover all weak equivalence classes of objects, and all morphisms between these objects, then we must have an equivalence on $K$-theory.  More formally, we have the following:

\begin{lemma} \lbl{lem:level0eq}
Suppose that $\C$ has sufficiently many covers, and that $\SC(\C)$ sits inside $\SC(\D)$ as a full subcategory.  Then the induced map $|w\SC(\C)| \rightarrow |w\SC(\D)|$ is a homotopy equivalence.
\end{lemma}

\begin{proof}
Using Quillen's Theorem A (from \cite{quillen73}) it suffices to show that for all $Y\in w\SC(\D)$ the category $(Y\downarrow w\SC(\C))$ is contractible.  Now as $(Y\downarrow w\SC(\D))$ is a preorder (by lemma \ref{lem:cofiltpre}) and $\SC(\C)$ is a subcategory of $\SC(\D)$, we know that $(Y\downarrow w\SC(\C))$ is also a preorder; thus to show that it is contractible we only need to know that it is cofiltered.  In addition, as $\SC(\C)$ is a full subcategory of $\SC(\D)$, it in fact suffices to show that we have enough objects for it to be cofiltered, so it suffices to show that this category is nonempty for all $Y$.

So let us show that for all $Y\in w\SC(\D)$ the category $(Y\downarrow w\SC(\C))$ is nonempty. We need to show that for any $Y\in w\SC(\D)$ there exists a $Z\in w\SC(\C)$ an a weak equivalence $\bd Y & \rWeakEquiv& Z\ed$.  Write $Y = \SCob{y}{i}$.  For each $i\in I$, let $\{\bd y^{(i)}_\alpha & \rSub & y_i\ed\}_{\alpha\in A_i}$ be the cover guaranteed by the sufficient covers condition, and let $\beta^{(i)}_\alpha:\bd y^{(i)}_\alpha & \rTo & z^{(i)}_\alpha\ed$ be the horizontal isomorphisms guaranteed by the sufficient covers condition.  Then the induced vertical morphism $\bd \{y^{(i)}_\alpha\}_{i\in I,\alpha\in A_i} & \rCover & \SCob{y}{i}\ed$ is a covering sub-map and the vertical morphism $\beta:\bd\{y^{(i)}_\alpha\}_{i\in I,\alpha\in A_i}&\rTo & \{z^{(i)}_\alpha\}_{i\in I,\alpha\in A_i}\ed$ is a horizontal isomorphism, so the morphism in $\SC(\D)$ represented by this is a weak equivalence.  But by definition $\{z^{(i)}_\alpha\}_{i\in I,\alpha\in A_i}$ is in $\SC(\C)$, so we are done.
\end{proof}

In the statement of the previous lemma we had two conditions.  One was a condition on $\C$, and one was a condition on $\SC(\C)$.  We would like to get those conditions down to conditions just about $\C$, as that will make using this kind of results easier.  In order for a morphism of $\SC(\D)$ to be in $\SC(\C)$ we need some representative of the morphism to come from a diagram in $\Tw(\C)$; in particular, this means that both the representing object, and the morphisms which are the components of the vertical and horizontal components, must be in $\C$.

If $\C$ is not a full subcomplex of $\D$ then much of this analysis becomes much more difficult, so for the rest of this section we will assume that $\C$ is a full subcomplex of $\D$.  This means that as long as we know that a representing object of the morphism is in $\Tw(\C)$, it is sufficient to conclude that the morphism will be in $\SC(\C)$.  In particular, we want to be able to conclude that just because the source and target of a morphism are in $\SC(\C)$ then the morphism must be, as well.  We can translate this into the following condition.

\begin{definition}  Suppose that $\C$ is a full subcomplex of $\D$.
We say that $\C$ is \textsl{wide} (respectively, \textsl{tall}
  if for any horizontal (resp. vertical) morphism $\bd A & \rTo & B \ed\in \D$,
  if $B$ is in $\C$ then so is $A$.
\end{definition}

If $\C$ is a full subcomplex of $\D$ then we know that $\Tw(\C)$ is a full subcategory of $\Tw(\D)$. If $\C$ happens to also be wide, we know something even stronger: given any horizontally connected component of $\Tw(\D)$, either that entire component is in $\Tw(\C)$, or nothing in the component is in $\Tw(\C)$.  Analogously, if $\C$ is tall we can say the same thing for vertically connected components.  This lets us conclude that $\SC(\C)$ is a full subcategory of $\SC(\D)$.

\begin{lemma}
Let $\C$ be a full subcomplex of $\D$.  If $\C$ is wide or tall then $\SC(\C)$ is a full subcategory of $\SC(\D)$.
\end{lemma}

\begin{proof}
Let $\SCob{a}{i}, \SCob{b}{j}\in \SC(\C)$, and let $f:\SCob{a}{i}\rightarrow \SCob{b}{j}$ be a morphism in $\SC(\D)$.  This morphism is represented by a diagram
\begin{diagram}
\SCob{a}{i} & \lSub^p & \SCob{a'}{k} & \rTo^\sigma & \SCob{b}{j}
\end{diagram}
In order for $f$ to be in $\SC(\C)$ it suffices to show that each $a'_k$ is in $\C$, as $\C$ is a full subcategory of $\D$.  Now if $\C$ is wide then for all $k\in K$ we have a horizontal morphism $\Sigma_k:\bd a_k & \rTo & b_{\sigma(k)}\ed$.  As $\C$ is wide and each $b_j\in \C$ we must have $a'_k\in \C$ for all $k$.  So $\SC(\C)$ is a full subcategory of $\SC(\D)$.  If, on the other hand, $\C$ is tall then for each $k\in K$ we consider the vertical morphism $P_k:\bd a'_k & \rSub & a_{p(k)}\ed$.  As $a_i\in \C$ for all $i\in I$ we must also have $a'_k\in \C$ for all $k'\in K$.  So $\SC(\C)$ is a full subcategory of $\SC(\D)$, and we are done.
\end{proof}

Which leads us to the following approximation result.

\begin{proposition} \lbl{lem:widetalls.}
Suppose that $\C$ is a subcomplex of $\D$ with sufficiently many covers.  If $\C$ is wide or tall, the inclusion $\C\rightarrow \D$ induces an equivalence $K(\C)\rightarrow K(\D)$.
\end{proposition}

\begin{proof}
Lemma \ref{lem:level0eq} shows that $K(\C)_i\rightarrow K(\D)_i$ is an
equivalence for $i=0$.  If we can show that for all $n$, $s_n\C$ is a wide or
tall subcomplex of $s_n\D$ with sufficiently many covers we will be done, as we
will be able to induct on $i$ to see that the induced morphism is an equivalence
on all levels.  In fact, note that it suffices to show that $f_n\C$ is a wide
(resp.  tall) subcomplex of $f_n\D$ with sufficiently many covers.

First we show that $f_n\C$ has sufficiently many covers in $f_n\D$.  Consider an
object $D$
\begin{diagram}
  D_1 & \lCover & \cdots & \lCover & D_n
\end{diagram}
of $f_n\D$.  As $\C$ has sufficiently many covers in $\D$ there exists a
covering family $\{\bd B_\alpha & \rSub & D_n\ed\}_{\alpha\in A}$ of $D_n$ in
which every object is horizontally isomorphic to an object of $\C$.  Given an
object $X\in \D$, let $\bar X\in f_n\D$ be the constant object where $\bar X_k =
\{X\}$.  Then the family $\{\bd \bar B_\alpha & \rSub & D\ed\}_{\alpha\in A}$ is
a covering family of $D$.  As each $B_\alpha$ was horizontally isomorphic to an
object in $\C$, each $\bar B_\alpha$ is horizontally isomorphic to something in
$f_n\C$, and we are done.

The fact that if $\C$ was a tall (resp. wide) subcomplex of $\D$ then $f_n\C$ is
a tall (resp. wide) subcomplex of $\D$ follows directly from the definition of
$f_n\C$ and $f_n\D$.
\end{proof}

Finally, we can generalize this result to simplicial polytope complexes.

\begin{corollary} 
Suppose that $\C_\dot\rightarrow \D_\dot$ is a morphism of simplicial polytope complexes.  If for each $n$, the morphism $\C_n\rightarrow \D_n$ is an inclusion of $\C_n$ as a subcomplex into $\D_n$ and satisfies the conditions of lemma \ref{lem:widetalls.}, then the induced map $K(\C_\dot)\rightarrow K(\D_\dot)$ is an equivalence.
\end{corollary}

We finish up this section with a couple of applications of this result.

\subsection*{More explicit formula for suspensions and cofibers.}

For any polytope complex $\C$ and any positive integer $n$ we have a polytope functor $\C\rightarrow f_n\C$ given by including and object $a$ as the constant object
$$\bar a = \begin{diagram}
\{a\} & \lCover & \{a\} & \lCover & \cdots & \lCover & \{a\}
\end{diagram}$$
This includes $\C$ as a wide subcomplex of $f_n\C$.  In fact, $\C$ also has sufficiently many covers.  Given any object 
$$A = \begin{diagram}
A_1 & \lCover & A_2 & \lCover & \cdots & \lCover & A_n
\end{diagram},$$
write $A_n = \SCob{a}{i}$.  Then the family $\{\bd \bar a_i & \rSub & A\ed\}_{i\in I}$ is a covering family, and each $\bar a_i\in \C$.  Thus we have an inclusion $\C^{\vee n} \rightarrow s_n\C$ which induces an equivalence on $K$-theory.

In fact, this is an equivalence on the $K$-theory of simplicial polytope complexes, as this inclusion commutes with the simplicial structure maps.  Thus $s_\dot\C$ can be considered to be a bar construction on $\C$, as the structure maps of $s_\dot\C$, when restricted to the constant objects, exactly mirror the morphisms of the bar construction.  (The $0$-th face map forgets the first one, the next $n-1$ glue successive copies of $\C$ together, and the $n$-th one forgets the last one, exactly as the bar construction does.  The degeneracies each skip one of the $\C$'s in $s_{n+1}\C$.)

Generalizing to simplicial polytope complexes, this gives the following simplifications of the formulas for $\sigma \C_\dot$ and $(\D/g).$ from corollary \ref{cor:sigmaA} and proposition \ref{prop:cofibers}:

\begin{corollary} \lbl{cor:comp}
  Let $g:\C_\dot\rightarrow \D_\dot$ be a morphism of simplicial polytope complexes.  Let $\sigma \C_\dot$ and $(\D/g).$ be the simplicial polytope complexes defined by 
  \[(\sigma\C_\dot)_n = \C_n^{\vee n} \qquad \hbox{and} \qquad (\D/g)_n = \D_n \vee \C_n^{\vee n}.\]
  Then $\Omega K(\sigma\C_\dot) \simeq K(\C_\dot)$ and 
  \[K(\C_\dot) \longrightarrow K(\D_\dot) \longrightarrow K((\D/g).)\]
  is a cofiber sequence of spectra.
\end{corollary}

It is necessary to check that these inclusions commute with the simplicial maps, but it is easy to see that they do.  Note that on $(\D/g)_n$, $\partial_0$ is induced by the three morphisms
$$\partial_0:\D_n\rightarrow \D_{n-1} \qquad g\partial_0:\C_n\rightarrow \D_{n-1} \qquad \partial_0^{\vee n-1}:\C_n^{\vee n-1} \rightarrow \C_{n-1}^{\vee n-1}.$$

\subsection*{Local data on homogeneous manifolds.}

Let $X$ be a geodesic $n$-manifold with a preferred open cover $\{U_\alpha\}_{\alpha\in A}$ such that for any $\alpha\in A$ and any two points $x,y\in U_\alpha$ there exists a unique geodesic connecting $x$ and $y$.  (For example, $X = E^n$, $S^n$, or $\mathbb{H}^n$ are examples of such $X$.  In the first and third case we take our open cover to be the whole space; in the second case we take it to be the set of open hemispheres.)  We then define a polytope complex $\C_X$ in the following manner.  Define a \textsl{simplex} of $X$ to be a convex hull of $n+1$ points all sitting inside some $U_\alpha$ with nonempty interior, and a \textsl{polytope} of $X$ to be a finite union of simplices.  We then define $\C_{Xv}$ to be the poset of polytopes of $X$ under inclusion with the obvious topology.  Given two polytopes $P$ and $Q$, we define a \textsl{local isometry of $P$ onto $Q$} to be a triple $(U,V,\varphi)$ such that $U$ and $V$ are open subsets of $X$ with $P\subseteq U$ and $Q\subseteq V$, $\varphi:U\rightarrow V$ is an isometry of $U$ into $V$, and $\varphi(P) = Q$.  Then we define a horizontal morphism $\bd P & \rTo & Q\ed$ to be an equivalence class of local isometries of $P$ onto $Q$, with $(U,V,\varphi)\sim(U',V',\varphi')$ if $\varphi|_{U\cap U'} = \varphi'|_{U\cap U'}$.  Under these definitions it is clear that $\C_X$ is a polytope complex.

Now let $U\subseteq X$ be any preferred open subset of $X$ with the preferred cover $\{U\}$.  Then $\C_U$ is also a polytope complex and we have an obvious inclusion map $\C_U\rightarrow \C_X$.

\begin{lemma}
If $X$ is homogeneous then $\C_U\rightarrow \C_X$ induces an equivalence $K(\C_U)\rightarrow K(\C_X)$.
\end{lemma}

\begin{proof}
Clearly $\C_U$ is a tall subcomplex of $\C_X$.  Given any polytope $P\subseteq \C_X$ we can triangulate it by triangles small enough to be in a single chart.  Once we are in a single chart we can subdivide each triangle by barycentric subdivision until the diameter of every triangle in the triangulation is small enough that the triangle can fit inside $U$.  As $X$ is homogeneous there is a local isometry of any such triangle into $U$, and thus $\C_U$ has sufficiently  many covers.  Thus by proposition \ref{lem:widetalls.} the induced map $K(\C_U)\rightarrow K(\C_X)$ is an equivalence.
\end{proof}

Any isometry $X\rightarrow Y$ which takes preferred open sets into preferred open sets induces a polytope functor $\C_X\rightarrow \C_Y$ (which is clearly an isomorphism).  Thus the statement of proposition \ref{prop:XsimeqY} is exactly that all morphisms in the diagram
\[\C_X \longleftarrow \C_U \xrightarrow{\C_\varphi} \C_V \longrightarrow \C_Y\]
are equivalences, which follows easily from the above lemma.

\subsection*{Non-examples}

We conclude this section with a couple of non-examples.  First, take any polytope complex $\C$ and consider the polytope complex $\C\vee\C$.  $\C$ sits inside this (as the left copy, for example) and is tall by definition, but the K-theories of these are not equivalent as the left copy of $\C$ does not contain sufficiently many covers.  (In particular, it can't cover anything in the right copy of $\C$.)  However, if we added ``twist'' isomorphisms --- horizontal isomorphisms between corresponding objects in the left and right copies of $\C$ --- then the left $\C$ would contain sufficiently many covers, and the K-theories of these would be equal.

As our second non-example we will look at ideals of a number field.  Let $K$ be a number field with Galois group $G$.  Let the objects of $\C$ be the ideals of $K$.  We will have a vertical morphism $\bd I & \rSub & J\ed$ whenever $I | J$, and we will have our horizontal morphisms induced by the action of $G$.   The $K$-theory of this will be countably many spheres wedged together, one for each prime power ideal of $K$.  (See \cite{zakharevich10} section 5 for a more detailed exploration of this example.)  The prime ideals sit inside $\C$ as a wide subcomplex, but they do not give an equivalence because if $\mathfrak{p}^k$ is a prime power ideal for $k>1$ then it can't be covered by prime ideals.  The K-theories of these two will in fact be equivalent, since they are both countably many spheres wedged together, but the inclusion does not induce an equivalence.

\appendix

\section{The $S.$ construction} \lbl{app:S.} 

In this appendix we formally introduce the definitions and concepts first mentioned in section \ref{sec:S.intro}.  For more details on the basic definition of the $S_\dot$-construction, see \cite{waldhausen83}.  For the rest of this section, $\E$ is a Waldhausen category.

\begin{lemma}
Suppose that $\E'$ is a Waldhausen subcategory of $\E$, and $\E'$ contains all weak equivalences and cofibrations of $\E$.  Then the inclusion $\E'\rightarrow \E$ induces the identity morphism on $K$-theory.
\end{lemma}

\begin{proof}
This is true by simple observation of the definition of the $K$-theory of a Waldhausen category.  On a Waldhausen category, the $S_\dot$-construction uses only cofibrations in the definitions of the objects.  As the cofibrations in $S_\dot\E$ are in particular levelwise cofibrations, this means that for all $n\geq0$, in $S_\dot^{(n)}\E$ all morphisms in every diagram representing an object will be either cofibrations or cofiber maps.  Thus all of the objects of $S_\dot^{(n)}\E$ will be objects of $S_\dot^{(n)}\E'$.

In order to obtain the $n$-th space of $K(\E)$ we look at the geometric realization of $wS_\dot^{(n)}\E$.  Every $k$-simplex of this consists of a diagram, each of whose morphisms is either a cofibration, cofiber map, or weak equivalence.  We know that all weak equivalences of $\E$ are in $\E'$, and thus every simplex of $K(\E)_n$ is in $K(\E')_n$, which means that the natural inclusion is actually the identity morphism, as desired.
\end{proof}

Note that there exists a minimal simplification of $\E$, as given any family of simplifications $\{\E_\alpha\}_{\alpha\in A}$, the category $\bigcap_{\alpha\in A} \E_\alpha$ will also be a simplification of $\E$.

We now turn our attention to a result about when a pair of adjoint functors between Waldhausen categories induces a homotopy equivalence between the K-theories.

\begin{lemma}
We say that $F:\E\rightleftarrows \E':G$ is an \textsl{exact adjunction} if $F$ is left adjoint to $G$, both $F$ and $G$ are exact, and the unit and counit are objectwise weak equivalences.  An exact adjunction induces an adjoint pair of functors $wF:w\E\rightleftarrows w\E':wG$, and for all $n\geq 0$ the adjunction $S_nF:S_n\E \rightleftarrows S_n\E':S_nG$ is also an exact adjunction.
\end{lemma}

In such a case we sometimes say that $F$ is \textsl{exactly adjoint} to $G$.  When such an adjunction is an equivalence, we call it an \textsl{exact equivalence}.  Note that any equivalence which is exact in both directions is an exact equivalence, as all isomorphisms are weak equivalences.

\begin{proof}
  As $F$ and $G$ are exact we know that $wF$ and $wG$ are well-defined. In order to see that they are adjoint, note that the existence of a unit and counit are sufficient; as the unit and counit are natural weak equivalences they pass to natural transformations inside $w\E$ and $w\E'$, and thus give us the adjunction, as desired.

Now we need to show that an exact adjunction induces an exact adjunction on $S_n$.  As an exact functor passes to an exact functor on the $S_n$-level all that we must show is that the two functors $S_nF$ and $S_nG$ will be adjoint.  However, as both $S_n\E$ and $S_n\E'$ are diagram categories, with $S_nF$ and $S_nG$ defined levelwise, the adjunction follows directly from the adjunction between $F$ and $G$.  (The unit and counit will be defined levelwise.  So we are done.  
\end{proof}

This lemma implies that for every $n$ we get an induced pair of adjoint functors
$$wS_n^{(n)} \E \leftrightarrows wS_n^{(n)}\E',$$
and thus a levelwise homotopy equivalence between the $K$-theory spectra.  Thus we can conclude the following corollary:

\begin{corollary}
  An exact adjunction induces a homotopy equivalence between the $K$-theory spectra of the Waldhausen categories.
\end{corollary}

Lastly we have a computational result which allows us to simplify the categories that we study when we try to compute with the $S_\dot$-construction.

\begin{lemma} \lbl{lem:subobjects}
Suppose that $\E'$ is a subcategory of $\E$ with the property that given any morphism $f:A\rightarrow B$ in $\E$, there exists a factorization $f = hg$ where $h$ is an isomorphism and $g\in \E'$.  If we let $\tilde S_n\E$ be the full subcategory of $S_n\E$ containing all those objects from $S_n\E'$ then $\tilde S_n\E$ is exactly equivalent to $S_n\E$.
\end{lemma}

Note that $\E'$ automatically inherits a Waldhausen structure from $\E$.

\begin{proof}
It suffices to show that every object of $S_n\E$ will be isomorphic to an object from $S_n\E'$.  The condition on $\E'$ ensures that $\E'$ contains all objects of $\E$, as for any object $A\in \E$ if we factor the identity morphism into $hg$ as given in the statement, $g\in \E'$ which means that $A\in \E'$.

Note that it suffices to show that we can replace the longest row of cofibrations by cofibrations in $\E'$, as any two objects of $S_n\E$ which are equal on the first line are isomorphic.  As $\E'$ is a Waldhausen category, if we have an object of $S_n\E$ with first row from $S_n\E'$, we must have some object in $S_n\E'$ which is isomorphic to it.  Thus it now remains to show that given a diagram
\begin{diagram}
A_1 & \rCofib^{\iota_1} & A_2 & \rCofib^{\iota_2} & \cdots & \rCofib^{\iota_{n-1}} & A_n
\end{diagram}
in $\E$ there exists an isomorphism of diagrams to such a diagram in $\E'$.

We will show that given such a diagram where $\iota_1,\ldots,\iota_{k-1}\in \E'$ there exists an isomorphic diagram where $\iota'_1,\ldots,\iota'_k$ are in $\E'$.  The base case where $k=1$ is obvious.  Assuming that we have the case for $k-1$, factor $\iota_k$ into $g:A_k\rightarrow A'$ and $h:A'\rightarrow A_{k+1}$.  The following diagram shows that we have the case for $k$:
\begin{diagram}[small]
A_1 & \rCofib^{\iota_1} & \cdots & \rCofib^{\iota_{k-1}} & A_k & \rCofib^g & A' & \rCofib^{\iota_{k+1}h} & A_{k+2} & \rCofib^{\iota_{k+2}} & \cdots & \rCofib^{\iota_{n-1}} & A_n \\
\dEqual && && \dEqual && \dTo^h_\cong && \dEqual &&&& \dEqual \\
A_1 & \rCofib^{\iota_1} & \cdots & \rCofib^{\iota_{k-1}} & A_k & \rCofib^{\iota_k} & A_{k+1} & \rCofib^{\iota_{k+1}} & A_{k+2} & \rCofib^{\iota_{k+2}} & \cdots & \rCofib^{\iota_{n-1}} & A_n 
\end{diagram}
So we are done.
\end{proof}

\section{Proof of lemma \ref{lem:Lnstructure}} \lbl{app:technical}

This section concerns the proof of lemma \ref{lem:Lnstructure}:
\begin{lemLnstructure}
$L_n\SC(\C)$ is a Waldhausen category which is a simplification of $F_n\SC(\C)$.  The cofibrations (resp. weak equivalences) in $L_n\SC(\C)$ are exactly the morphisms which are levelwise cofibrations (resp. weak equivalences).
\end{lemLnstructure}

A morphism $A\rightarrow B \in F_n\SC(\C)$ is represented by a diagram
\begin{diagram}[small]
A_1 & \rCofib & A_2 & \rCofib & \cdots & \rCofib & A_n\\
\dTo && \dTo && && \dTo \\
B_1 & \rCofib & B_2 & \rCofib &\cdots & \rCofib & B_n
\end{diagram}
and by lemma \ref{lem:layering}(1) will be layered exactly when for each $i=1,\ldots,n-1$ the diagram
\begin{diagram}[small]
A_i/A_{i-1} & \rCofib & A_i \\
\dTo & & \dTo \\
B_i/B_{i-1} & \rCofib & B_i
\end{diagram}
commutes.  As each square is considered separately, for all of the proofs in this section we will assume that $n=2$, as for all other values of $n$ the proofs will be equivalent, and it saves on having an extra variable floating around.

\begin{lemma} \lbl{lem:levelcofib}
Any layered morphism which is levelwise a cofibration is a cofibration.
\end{lemma}

\begin{proof}
We want to show that if 
$$f_1,f_2: \begin{diagram} (\bd A_1 & \rCofib^i & A_2\ed) & \rCofib & (\bd B_1 & \rCofib^j & B_2\ed)
\end{diagram}$$
is layered, then the induced morphism $\varphi:A_2\cup_{A_1} B_1 \rightarrow B_2$ is a cofibration.  We know that $A_2\cup_{A_1}B_1 \cong (A_2/A_1)\amalg B_1$, and that $\varphi = j \amalg (f_2/f_1)$ (where the second part follows directly from the layering condition).  Thus it suffices to show that $f_2/f_1$ is a cofibration.  This follows from the more general statement that in $\SC(\C)$, given two composable morphisms $g,h$, if $h$ and $hg$ are cofibrations then so is $g$.  As $\bd A_2/A_1 & \rCofib & B_2\ed$ and $\bd B_2/B_1 & \rCofib & B_2\ed$ are both cofibrations, it follows that $f_2/f_1$ must be one as well.
\end{proof}

Now we turn our attention to showing that $L_n\SC(\C)$ is a simplification of $F_n\SC(\C)$.  We first develop a little bit of computational machinery for layering, which will allow us to work with cofibrations more easily.

Given any object $A=\SCob{a}{i}\in \SC(\C)$, we say that $A'$ is a \textsl{subobject} of $A$ if $A' = \{a_i\}_{i\in I'}$ for some subset $I'\subseteq I$.  If $A',A''$ are two subobjects of $A$, we will write $A'\cap A''$ for $\{a_i\}_{i\in I'\cap I''}$, and we will write $A'\subseteq A''$ if $I' \subseteq I''$.  Suppose that $f: \bd A &\rTo& B\ed$ is a morphism in $\SC(\C)$.  Pick a representation of this by a sub-map $p$ and a shuffle $\sigma$, and write $B = \SCob{b}{j}$.  Then $\im_BA = \{b_j\}_{j\in \im \sigma}$.  Note that this agrees with the previous definition of image when $f$ is a cofibration, and $\im_BA$ is a subobject of $B$.  If we write $A = A_1\amalg A_2$ then $A_1$ and $A_2$ are subobjects of $A$, and $\im_BA = \im_BA_1 \cup \im_BA_2$.  If $f$ were a cofibration, we also have $\im_BA_1 \cap \im_BA_2 = \emptyset$; if $f$ were a weak equivalence then $\im_BA = B$.  (For example, $\im_BA \cap (B/A) = \emptyset$.)  Given a second morphism $g:\bd B & \rTo & C\ed$, $\im_CA \subseteq \im_CB$.

Now consider a commutative square
$$f_1,f_2:\begin{diagram} (\bd A_1 & \rCofib & A_2\ed) & \rTo & (\bd B_1 & \rCofib & B_2\ed)\end{diagram}.$$
This square satisfies the layering condition exactly when $\im_{B_2}(A_2/A_1) \subseteq \im_{B_2}(B_2/B_1) = B_2/B_1$, or equivalently when $\im_{B_2}(A_2/A_1) \cap \im_{B_2}B_1 = \emptyset$.  We will use this restatement in our computations.

\begin{lemma} \lbl{lem:cofiblayered}
Cofibrations are layered.
\end{lemma}

\begin{proof}
If $\bd A & \rCofib & B\ed$ is a cofibration, then by definition $\bd A_2\cup_{A_1}B_1 & \rCofib & B_2\ed$ is a cofibration.  But we have an acyclic cofibration $\bd (A_2/A_1) \amalg B_1 & \rAcycCofib & A_2\cup_{A_1} B_1\ed$, so $\im_{B_2}(A_2/A_1) \cap \im_{B_2} B_1 = \emptyset$, as desired.
\end{proof}

\begin{lemma} \lbl{lem:layeredpushout}
Layered morphisms are closed under pushouts.  More precisely, given any commutative square 
\begin{diagram}[small]
A & \rCofib & B \\
\dTo & & \dTo \\
C & \rTo & D
\end{diagram}
in which all morphisms are layered, the induced morphisms
\begin{diagram}
C & \rCofib & B\cup_A C &\qquad\ \qquad &B & \rTo & B\cup_A C & \qquad\ \qquad & B\cup_A C & \rTo & D
\end{diagram}
are all layered.
\end{lemma}

\begin{proof}
The first of these is clearly layered as it is a cofibration.  

Write $X_i = B_i\cup_{A_i}C_i$.  Keep in mind that for all $i$, we have an acyclic cofibration $\bd (B_i/A_i)\amalg C_i & \rAcycCofib & X_i\ed$.

In order to show that the second is layered we need to show that $\im_{X_2}(B_2/B_1) \cap \im_{X_2}(X_1) = \emptyset$.  We have
\begin{eqnarray*}
\im_{X_2}(B_2/B_1) \cap \im_{X_2}X_1 &=& \im_{X_2}(B_2/B_1) \cap (\im_{X_2}C_1 \cup \im_{X_2}(B_1/A_1)) \\
&=& (\im_{X_2}(B_2/B_1) \cap \im_{X_2}C_1) \cup (\im_{X_2}(B_2/B_1) \cap \im_{X_2}(B_1/A_1)).
\end{eqnarray*}
Consider the first of the two sets we are unioning.  By the definition of $X_2$, $\im_{X_2}C_2 \cap \im_{X_2}B_2 = \im_{X_2}A_2$.  Thus 
$$\im_{X_2}(B_2/B_1) \cap \im_{X_2}C_1 \subseteq \im_{X_2}(B_2/B_1) \cap \im_{X_2}C_2 \subseteq \im_{X_2}A_2.$$
On the other hand, 
$$\im_{X_2}(C_1) \cap \im_{X_2}(A_2) = \im_{X_2}(\im_{C_2}C_1 \cap \im_{C_2}A_2) = \im_{X_2}(\im_{C_2}A_1) = \im_{X_2}A_1$$
as $\bd A & \rTo & C\ed$ is layered.  Thus we want to show that $\im_{X_2}(B_2/B_1) \cap \im_{X_2}A_1 = \emptyset$.  It suffices to show this inside $B_2$, where it is obvious.  Now consider the second part. As $\bd B_1/A_1 & \rTo & X_2\ed$ and $\bd B_2/B_1 & \rTo & X_2\ed$ both factor through $B_2$, it suffices to show that $(B_2/B_1) \cap \im_{B_2}(B_1/A_1) = \emptyset$, which is clear by definition.

It remains to show that the last of these morphisms is layered.  In particular, we need to show that
$$\im_{D_2}(X_2/X_1)\cap \im_{D_2}(D_1) = \emptyset.$$
But it is easy to see that
\begin{eqnarray*}
\im_{D_2}(X_2/X_1) \cap \im_{D_2}(D_1) &=& \im_{D_2}((C_2 \amalg B_2/A_2)/(C_1 \amalg B_1/A_1)) \cap \im_{D_2}(D_1) \\
&=& \im_{D_2}(C_2/C_1 \amalg (B_2/A_2)/(B_1/A_1)) \cap \im_{D_2}(D_1) \\
&=& (\im_{D_2}(C_2/C_1) \cap \im_{D_2}(D_1)) \cup (\im_{D_2}((B_2/A_2)/(B_1/A_1)) \cap \im_{D_2}(D_1)).
\end{eqnarray*}
the first of these is empty because $C\rightarrow D$ is layered.  The second is empty because $\im_{D_2}((B_2/A_2)/(B_1/A_1)) \subseteq \im_{D_2}(B_2/B_1)$, and the intersection of this with $\im_{D_2}(D_1)$ is empty because $\bd B& \rTo & D\ed$ is layered.  So we are done.
\end{proof}
    
Now we are ready to prove lemma \ref{lem:Lnstructure}.

\begin{proof}[Proof of lemma \ref{lem:Lnstructure}]
Firstly we will show that all weak equivalences of $F_n\SC(\C)$ are layered.  In particular, it suffices to show that any weak equivalences of $F_n\SC(\C)$ is also a cofibration, since we already know by lemma \ref{lem:cofiblayered} that all cofibrations are layered.  In particular, if we have a a commutative square
\begin{diagram}[large]
(\bd A_1 & \rCofib & A_2\ed) & \rWeakEquiv & (\bd B_1 & \rCofib & B_2\ed)
\end{diagram}
we want to show that the induced morphism $\bd A_2\cup_{A_1} B_1 & \rTo & B_2\ed$ is a cofibration.  As weak equivalences are preserved under pushouts we know that $\bd A_2 & \rWeakEquiv  & A_2\cup_{A_1} B_1\ed$ is a weak equivalence, as is $\bd A_2 & \rWeakEquiv & B_2\ed$.  As weak equivalences satisfy 2-of-3 we are done.

All morphisms $\bd A & \rTo & *\ed$ are in $L_n\SC(\C)$, as these are trivially layered.  As lemma \ref{lem:layeredpushout} showed that $L_n\SC(\C)$ is closed under pushouts, we see that $L_n\SC(\C)$ is, in fact, a simplification of $F_n\SC(\C)$.  Weak equivalences of $L_n\SC(\C)$ are levelwise because weak equivalences in $F_n\SC(\C)$ are levelwise, and cofibrations are levelwise by lemma \ref{lem:levelcofib}.
\end{proof}

\bibliographystyle{plain}
\bibliography{IZ}

\end{document}